\documentclass[11pt]{amsart}
\usepackage[top=2cm, left=2cm, right=2cm, bottom=2cm]{geometry}
\usepackage{amsfonts}
\usepackage{amsthm}
\usepackage{amsmath}
\usepackage{amssymb}
\usepackage{latexsym}
\usepackage{epic}
\usepackage{verbatim}
\usepackage{mathpazo} 
\usepackage{graphicx}

\usepackage[T1]{fontenc} 
\usepackage[parfill]{parskip} 

\usepackage{xcolor}

\usepackage{hyperref}
\hypersetup{colorlinks=true, urlcolor=blue, citecolor=blue, linkcolor=blue}
\usepackage{url}

\def\+{\!+\!}
\def\-{\!-\!}

\def\m{\!-\!}
\def\L{\mathcal{L}}

\def\bproc{{\bf procedure\ }}
\def\com#1{{\bf*}\hspace{6pt}{\sl #1}}
\def\la{\leftarrow}

\def\bfor{{\bf for\ }}
\def\bto{{\bf to\ }}
\def\bdo{{\bf do\ }}
\def\qq{\qquad}
\def\bif{{\bf if\ }}
\def\bthen{{\bf then\ }}
\def\belse{{\bf else\ }}
\def\belsif{{\bf elsif\ }}

\def\CPM{Annual Symp.\ Combinatorial Pattern Matching}

\def\JDA{J.\ Discrete Algorithms\ }

\def\TCS{Theoret.\ Comput.\ Sci.\ }

\theoremstyle{plain}
\newtheorem{theorem}{Theorem}
\newtheorem{lemma}[theorem]{Lemma}
\newtheorem{corollary}[theorem]{Corollary}

\newtheorem{remark}[theorem]{Remark}

\theoremstyle{definition}

\theoremstyle{remark}

\begin{document}

\title{More Properties of the Fibonacci Word on an Infinite Alphabet}

\author{Amy Glen}

\address{Amy Glen \newline
\indent School of Engineering \& Information Technology \newline
\indent Murdoch University \newline
\indent  90 South Street \newline
\indent Murdoch, WA 6150 AUSTRALIA}%
\email{\href{mailto:A.Glen@murdoch.edu.au}{A.Glen@murdoch.edu.au}}

\author{Jamie Simpson}

\address{Jamie Simpson \newline
\indent Department of Mathematics and Statistics \newline
\indent Curtin University \newline
\indent Bentley, WA 6102 AUSTRALIA}%
\email{\href{mailto:Jamie.Simpson@curtin.edu.au}{Jamie.Simpson@curtin.edu.au}}

\author{W. F. Smyth}
\address{W. F. Smyth \newline
\indent Department of Computing and Software \newline
\indent McMaster University \newline
\indent Hamilton, Ontario L8S4K1 CANADA}%
\email{\href{mailto:smyth@mcmaster.ca}{smyth@mcmaster.ca}}

\date{May 25, 2018}

\begin{abstract} %
Recently the Fibonacci word $W$ on an infinite alphabet was introduced by
[Zhang {\it et al.}, {\it Electronic J.\ Combinatorics} 24--2 (2017) \#P2.52]
as a fixed point of the morphism
$\phi: (2i) \mapsto (2i)(2i\+ 1),\ (2i\+ 1) \mapsto (2i\+ 2)$
over all $i \in \mathbb{N}$.
In this paper we investigate the occurrence of squares, palindromes, and Lyndon factors
in this infinite word.
\end{abstract}



\maketitle

\section{Introduction}
\label{sect-1}

A word of $n$ letters is $x=x[1 \dots n]$, with $x[i]$ being
the $i$-th letter and $x[i\dots j]$ the \emph{factor} consisting of letters from
position $i$ to position $j$.  If $i=1$ then the factor is a
\emph{prefix} and if $j=n$ it is a \emph{suffix}. The letters in $x$
come from some \emph{alphabet} $A$.  The \emph{length} of $x$, written $|x|$,
is the number of occurrences of letters in $x$ and the number of
occurrences of the letter $a$ in $x$ is denoted by $|x|_a$. Two or more
adjacent identical factors form a \emph{power}. A word $x$ or factor $x$ is
\emph{periodic} with period $p$ if $x[i]=x[i+p]$ for all $i$ such
that $x[i]$ and $x[i+p]$ are in $x$. A periodic word with least
period $p$ and length $n$ is said to have \emph{exponent} $n/p$. The
word $ababa$ has exponent $5/2$ and can be written as $(ab)^{5/2}$.
Thus powers have integer exponent
at least 2. A factor with exponent 2 is a \emph{square}. Two words
$x$ and $y$ are \emph{conjugate} if there exist words $u$ and $v$
such that $x=uv$ and $y=vu$.   If $x=x[1\dots n]$ then the
\emph{reverse} of $x$, written $R(x)$, is $x[n]x[n-1]\dots x[1].$ A
word that equals its own reverse is a \emph{palindrome}. If $x=uvu$
we say that $x$ has \emph{border} $u$, and we see that $x$ has
period $|x|-|u|$. If the alphabet is ordered then words are ordered lexicographically.  A word which is lexicographically less than each of it conjugates is a \emph{Lyndon word}. Lyndon words are necessarily {\it primitive} (i.e., not a power of a shorter word).

Recall that the \textit{Fibonacci word} $F$ over the binary alphabet $\{0,1\}$ is the fixed point of the morphism $\psi(0)=01$, $\psi(1)=0$, and begins $$F=010010100100101001010\cdots$$  Recently Zhang, Wen, and Wu \cite{ZWW17} introduced an interesting modification of this word.  As an alphabet they used the non-negative integers, and as a morphism they used $\phi(2i)=2i \cdot 2i+1$ and $\phi(2i+1)=2i+2$. Here and elsewhere we use "$\cdot$" to indicate concatenation when it may not be clear from the context.  From this we see that $\phi(0)=01$, $\phi^2(0)=012$, and so on, with a fixed point of $\phi$ being the infinite word beginning $$W=012232342344523445456\cdots$$ We will call this the \textit{ZWW word} after the authors of \cite{ZWW17} and give it the symbol $W$ as the majority of the names of the authors of \cite{ZWW17} begin with $W$. Some of the properties of the Fibonacci word have parallels with those of the ZWW word.  For example, if we reduce the elements of the ZWW word modulo 2 we obtain the Fibonacci word. The well-known \textit{finite Fibonacci words} are $F_0$, $F_1$, etc. where $F_i=\psi^i(0)$.  The first few of these words are shown in Table~\ref{introtable1} below.  Note that these finite words have the property that
\begin{equation} \label{intro_e1}
F_{i+2}=F_{i+1}\cdot F_i \quad \mbox{for $i \geq 0$}.
\end{equation}
Analogously, we define the \textit{finite ZWW words} as $W_i=\phi^i(0)$.  The first few of these are also shown in Table~\ref{introtable1} below.

\begin{table}[htb!] \label{introtable1}
\begin{center}
\caption{Finite Fibonacci and ZWW words}
    \begin{tabular}{ll}
    $  F_0=0 $& $W_0=0 $\\
    $  F_1=01 $& $ W_1=01 $\\
    $  F_2=010$ & $W_2=012 $\\
    $  F_3=01001$ & $ W_3=01223 $\\
    $  F_4=01001010$ & $W_4=01223234 $\\
    $  F_5=0100101001001$ & $W_5=0122323423445 $\\
    \end{tabular}
  \end{center}
\end{table}
It is easily shown that, for $i \ge 0$,
\begin{equation} \label{introe0}
|W_i|=|F_i|=f_{i+2}
\end{equation} where $f_i$ is the $i$-th Fibonacci number defined by: $f_1=f_2=1$, $f_{i+2}=f_i+f_{i+1}$ for $i \ge 1$.

For a finite word $w$,  $n \oplus w$ denotes the word formed from $w$ by adding $n$ to each of its members.  Analogously to \eqref{intro_e1}, we have the following result.
\begin{lemma}\label{lemm-l2} For $i\ge 2$,
\begin{equation} \label{e1}
W_{i+1}=W_i \cdot 2 \oplus W_{i-1}.
\end{equation}
\begin{proof} We note that for any number $i$
$$\phi(2 \oplus i)=2 \oplus \phi(i).$$  We use induction on $i$.  It is easily checked that the
statement holds for $i=2$. Suppose it
holds for all $i$ with $2 \le i \le j$.  Then
\begin{eqnarray*}
W_{j+1}&=&\phi(W_j)\\
&=&\phi(W_{j-1}\cdot 2\oplus W_{j-2}) \text{ by the induction hypothesis}\\
&=& W_j \cdot \phi(2\oplus W_{j-2})\\
&=& W_j \cdot 2\oplus \phi(W_{j-2})\\
&=& W_j \cdot 2\oplus W_{j-1}
\end{eqnarray*} as required.
\end{proof}
\end{lemma}

\begin{lemma} \label{l1} For all $i \ge 0$,  the first letter of $W_i$ is $0$ and the last letter of $W_i$ is $i$. Moreover, $i$ occurs only once in $W_i$ and all the other letters in $W_i$ are less than $i$.
\end{lemma}
\begin{proof}
 This follows easily by induction using the previous lemma.
\end{proof}

We now present two ways of factorising $W_k$.  These will be used later.
\begin{lemma} \label{intro l2} For $k \ge 2$,
\begin{eqnarray*}
W_k &= &01\cdot \prod_{i=0}^{k-2}2\oplus W_i\\
&=& 01\cdot 2 \oplus \prod_{i=0}^{k-2}W_i
\end{eqnarray*} where $\prod$ indicates concatenation.
\begin{proof}The second and third parts of the display are clearly equal.  We will use induction on $k$ to show the first and second are.  The theorem clearly holds when $k=2$.  We assume it hold up to $k-1$ so that
$$W_{k-1}=01\cdot \prod_{i=0}^{k-3}2\oplus W_i.\\$$ Now using Lemma \ref{lemm-l2}, we have
\begin{eqnarray*}
W_k &=& W_{k-1} \cdot 2 \oplus W_{k-2}\\
&=&  01\cdot \Pi_{i=0}^{k-3}2 \oplus W_i \cdot 2 \oplus W_{k-2}\\
&=& 01\cdot \Pi_{i=0}^{k-2}2\oplus W_i
\end{eqnarray*}
and the result follows.
\end{proof}
\end{lemma} As an example of Lemma \ref{intro l2} consider,
\begin{eqnarray*}
W_5 & = & 01  \cdot 2 \oplus W_0 \cdot2 \oplus W_1 \cdot2 \oplus W_2 \cdot2 \oplus W_3\\
& =  & 01 \cdot 2 \oplus 0 \cdot2 \oplus 01 \cdot2 \oplus 012 \cdot2 \oplus 01223\\
& = &01\cdot 2\cdot 23 \cdot 234 \cdot 23445.
\end{eqnarray*}  The second factorisation also uses iteration of Lemma \ref{lemm-l2}.
\begin{lemma}
\label{lemm-W-2}
For $k \ge 1$,
\begin{eqnarray}
\label{odd}
W_{2k-1} &=& \Big[\prod\limits_{j=k}^1 2(k - j) \oplus W_{2j-2}\Big] \cdot (2k - 1), \\
\label{even}
W_{2k} &=& \Big[\prod\limits_{j=k}^1 2(k - j) \oplus W_{2j-1}\Big] \cdot (2k),
\end{eqnarray}
where $\prod\limits_{j=k}^1$ indicates concatenation
in the order $j = k,k-1,\ldots,1$.
\end{lemma}
\begin{proof}
We give the proof for $2k\- 1$; the proof for $2k$ is similar.
Since $W_1 = 0 \oplus W_0 \cdot 1 = 01$,
the result holds for $k = 1\ (2k \- 1 = 1)$.
Suppose it is true for some $2k\- 1,\ k\ge 1$.
Then
\begin{eqnarray*}
W_{2k+1} & = & W_{2k}\cdot 2\oplus W_{2k-1}, \mbox{\ by Lemma~\ref{lemm-l2}} \\
& = & W_{2k} \cdot 2\oplus \Big[\prod\limits_{j=k}^1 2(k\- j)\oplus W_{2j-2} \cdot (2k\- 1)\Big] \quad \mbox{\ by the induction hypothesis} \\
& = & W_{2k} \cdot  \Big[\prod\limits_{j=k}^1 2(k\+ 1\- j)\oplus W_{2j-2}\Big] \cdot (2k\+ 1) \\
& = & \Big[\prod\limits_{j=k+1}^1 2(k\+ 1\- j)\oplus W_{2j-2}\Big] \cdot (2k\+ 1),
\end{eqnarray*}
and so the result holds for all odd indices $2k\+ 1$.
\end{proof}
As an example of Lemma~\ref{lemm-W-2}, consider
\begin{eqnarray*}
W_5 &=& W_4 \cdot 2\oplus W_2 \cdot 4\oplus W_0 \cdot 5 \\
&=& 01223234 \cdot 2\oplus 012\cdot 4\oplus 0 \cdot 5 \\
&=& 01223234 \cdot 234 \cdot 4 \cdot 5.
\end{eqnarray*}

In \cite{ZWW17} the authors investigated the growth of $W$; in particular, they showed that the $n$-th letter in $W$ is less than $c \log n$ for some constant $c$ and that the sum of the letters of $W_k$ is
$$\frac{k(\gamma^{k+1}+\psi^{k+1})(\psi^2+\gamma)}{(\gamma-\psi)^2}+\frac{\psi^k-\gamma^k}{(\gamma-\psi)^3}$$ where $\gamma=\frac{1+\sqrt{5}}{2}$ and $\psi=\frac{1-\sqrt{5}}{2}$. Much of their paper was devoted to presenting various factorisations of $W$ using singular words.  They also showed that the only palindromes in $W$ are those of the form $2i \oplus 22$, $2i \oplus 232$, and $2i \oplus 323$ where $i$ is any non-negative integer.  In the present paper we are mainly concerned with the finite words $W_k$.  In Section~\ref{S:letters} we count the numbers of occurrences of letters in $W_k$, in Section~\ref{S:palindromes} we count palindromes, and in Sections~\ref{S:squares} and \ref{S:total-squares} we count squares.  In the final two sections we count the Lyndon factors in $W_k$ and describe the Lyndon array of $W_k$.
\section{Number of occurrences of a letter in $W_k$}
\label{S:letters}

We set $N(i,n)$ to be the number of occurrences of the letter $n$ in $W_i$. It is clear that
\begin{eqnarray}
\label{ee2} &&N(i,0)=1 \text{ for all } i \geq 0, \\
\label{ee3} &&N(i,1)=1 \text{ for all } i \ge 1,\\
\label{ee4} &&N(0,n)=0 \text{ for all } n \ge 1,\\
\label{ee5} &&N(1,n)=0 \text{ for all } n \ge 2.
\end{eqnarray}
\begin{theorem}
\begin{equation} \label{ee6} N(i,n)=\binom{i-n+\lfloor \frac{n}{2}\rfloor}{\lfloor \frac{n}{2}\rfloor},
\end{equation} where $\binom{x}{y}=0$ if $y>x$, $x<0$ or $y<0$.
\begin{proof}
 Note that each of \eqref{ee2}, \eqref{ee3}, \eqref{ee4} and \eqref{ee5} is satisfied by \eqref{ee6}. By Lemma~\ref{lemm-l2}, we have, for $i \ge 1$,
\begin{eqnarray*}
N(i+1,n)&=&|W_i \cdot 2 \oplus W_{i-1}|_n\\
&=& |W_i|_n + |2 \oplus W_{i-1}|_n\\
&=& |W_i|_n + | W_{i-1}|_{n-2}\\
&=& N(i,n)+N(i-1,n-2).
\end{eqnarray*}

Also, by Lemma \ref{l1}, $N(i,n)=0$ whenever $n>i$. Now we prove that \eqref{ee6} holds when $i \ge n$. We use induction on $i$.  This is true when $i=0$ and $i=1$ by (\ref{ee4}) and (\ref{ee5}). Suppose it holds for $i=1,\dots,k$. Then
\begin{eqnarray*}
N(k+1,n) &=& \binom{k-n+\lfloor \frac{n}{2} \rfloor}{\lfloor \frac{n}{2} \rfloor}+\binom{k-n+\lfloor \frac{n}{2} \rfloor}{\lfloor \frac{n}{2} \rfloor - 1}\\
&=& \binom{k-n+\lfloor \frac{n}{2} \rfloor+1}{\lfloor \frac{n}{2} \rfloor}
\end{eqnarray*} as required.
\end{proof}
\end{theorem}


\section{Palindromes} \label{S:palindromes}

Zhang \textit{et al.} \cite[Prop.~27]{ZWW17} showed that there are no non-trivial palindromes in the ZWW word other than $2i \oplus 22$, $2i \oplus 232$, and $2i \oplus 323$ for $i\ge 0$.  There are also trivial palindromes consisting of a single letter. In the next two theorems we count the total number of palindromes in $W_i$ and the number of distinct palindromes in $W_i$. By Lemma~\ref{lemm-l2}, we have $$W_{i}=W_{i-1} \cdot 2 \oplus W_{i-2} \quad \mbox{for $i \ge 0$}.$$ Palindromes in $W_i$ therefore come in three types: those contained in $W_{i-1}$, those contained in $2 \oplus W_{i-2}$ and those which straddle the boundary between  $W_{i-1}$ and  $2 \oplus W_{i-2}$.  If a palindrome is of the last type for some $i$ we say it is  \emph{straddling}.

\begin{lemma} \label{L:pal} The only straddling palindromes in the ZWW word are  $22$, $232$, and $323$ which occur only in $W_3$ and $W_4$.
\end{lemma}
\begin{proof} A straddling palindrome must include the last letter of $W_{i-1}$ and the first letter of $2 \oplus W_{i-2}$ for~$i \ge 2$. By Lemma \ref{l1} these are $i-1$ and 2 respectively.  Therefore both $2\cdot (i-1)$ and $(i-1) \cdot 2$ must appear in the palindrome. Since 2 is always followed by 2 or 3 in the ZWW word, we must have $i=3$ or $4$. When $i=3$ we have $W_i=01223$  and when $i=4$ we have $W_i=01223234$.  These contain the straddling palindromes 22, 232 and 323.  There can be no other straddling palindromes in the ZWW word.
\end{proof}
\begin{theorem} The total number of palindromes in $W_i$, including single letter palindromes,  is 1, 2, 3, 6 for $i=0,1,2,3$ respectively, and $f_{i+3}-2f_{i-2}$ for $i>3$.
\end{theorem}
\begin{proof} We write $P(W_i)$ and $S(W_i)$ for the total number of palindromes in $W_i$ and the number of straddling palindromes in $W_i$ respectively. It is easily checked that, for $i=0,1,2 \text{ and }3$,  $P(W_i)$ equals 1, 2, 3 and 6  respectively. Since the total number of palindromes in $W_i$ equals the total number in $2 \oplus W_i$ we get, from Lemma~\ref{lemm-l2}, that $$P(W_{i+2})=P(W_{i+1})+S(W_{i+2})+P(W_i).$$ From Lemma \ref{L:pal}, $S(W_i)=0$ when $i>4$, so for $i > 2$, we have $$P(W_{i+2})=P(W_{i+1})+P(W_i).$$ This recurrence gives $P(W_i)=f_{i+3}-2f_{i-2}$.
\end{proof}

We write $D(W_i)$ for the number of distinct palindromes in $W_i$.  It is easily checked that $D(W_i)=i+1$ for $i=$ 0, 1, 2 and $D(W_3)=5$.

\begin{theorem} For $i \ge 3$, we have \begin{equation} \label{eq:distinct-pal} D(W_i)=\lfloor 5i/2 \rfloor -2. \end{equation}
\end{theorem}
\begin{proof} Since $W_{i+2}=W_{i+1} \cdot 2 \oplus W_i$ and there are no straddling palindromes in $W_i$ when $i>4$, any palindrome appearing in $W_{i+2}$ for the first time must occur in $2 \oplus W_i$.  It must equal $2 \oplus p$ where $p$ is a palindrome that occurred for the first time in $W_i$.  From this we see that 22 and 3 first occur in $W_3$, 44 and 5 first in $W_5$, 66 and 7 first in $W_7$, and so on. Similarly 4, 232 and 323 occur for the first time in $W_4$, 6, 454 and 545 for the first time in $W_6$, and so on. From this we see that $D(W_{i+2})=5+D(W_i)$ for $i\ge 3$.  With the initial conditions $D(W_3)=5$ and $D(W_4)=8$ we obtain (\ref{eq:distinct-pal}).
\end{proof}

\section{Distinct Squares in $W_k$} \label{S:squares}

Squares in the finite Fibonacci words have been characterised in \cite{FS99} and \cite{IMS97}. In this section we characterise the distinct squares in $W_k$.  Our analysis here depends on the factorisation of $W_k$ obtained in Lemma \ref{lemm-l2}.  We will also need the following lemma.
\begin{lemma} Let $i$ and $k$ be integers such that $0 \le i \le \lfloor k/2 \rfloor$.  Then  the length $|W_{k-2i}|$ suffix of $W_k$ is $2i \oplus W_{k-2i}$.
\begin{proof} We use induction on $k$.  The lemma  holds vacuously when $k=0$ and $k=1$.  Suppose it holds up to $k-1$.  By Lemma \ref{lemm-l2},
$$W_k=W_{k-1} \cdot 2 \oplus W_{k-2}.$$ We see immediately that the lemma holds when $i=1$.  For $1 < i \le \lfloor k/2 \rfloor$ the length $|W_{k-2i}|$ suffix of $W_k$ is the length $|W_{k-2-2(i-1)}|$ suffix of $2 \oplus W_{k-2}$ which is $2 \; \oplus$ the length $|W_{k-2i}|$ suffix of $W_{k-2}$. By the induction hypothesis this is $$2 \oplus (2(i-1) \oplus W_{k-2-2(i-1)})=2i \oplus W_{k-2i}$$ and the lemma is proved.
\end{proof}
\end{lemma}
With Lemma \ref{lemm-W-2} this gives the following result.
\begin{lemma} \label{suffix factoristion}  For $k\ge 0$,
$$W_{k+1}=S_{k,k}\cdot S_{k,k-2}\cdot \dots \cdot S_{k,k-2\lfloor k/2 \rfloor}\cdot (k+1)$$ where $S_{k,i}$ is the length $|W_i|$ suffix of $W_k$.
\end{lemma}
As an example of the factorisation presented in the above lemma,  consider the following:
\begin{eqnarray*}
W_4&=&01223234\\
W_5
&=& 01223234 \cdot 234 \cdot 4 \cdot 5
\end{eqnarray*} We now prove the main result of this section. Since $W_k$ is a prefix of $W_{k+1}$ the set of squares in $W_{k+1}$ contains all those squares that appeared in $W_k$.  Other squares in $W_{k+1}$ we call \emph{new squares}.
\begin{theorem} \label{T:distinct-squares}
\label{lemm-distinct}
For every $k \ge 0$:
\begin{itemize}
\item[(a)]
$W_{k+1}$ introduces $\lfloor k/2 \rfloor$ new squares
of periods $f_{k},f_{k-2},\ldots,f_{k-2\lfloor k/2 \rfloor +2}$ \thinspace;
\item[(b)]
$W_{k+1}$ contains exactly $\lfloor k/2 \rfloor \lceil k/2 \rceil$ distinct squares.
\end{itemize}
\begin{proof} Consider the factorisation in the last lemma.  Since $S_{k,i+2}$ is a suffix of $S_{k,i}$, $W_{k+1}$ contains squares $S_{k,k-2i}^2$ for $i=1,\dots,\lfloor k/2 \rfloor.$ Each of these squares has period $|W_{k-2i}|$ which equals $f_{k-2i+2}$ for $i=1,\dots,\lfloor k/2 \rfloor$.  We will show that these are all new squares and that they are the only new squares in $W_{k+1}$.
Note that no new square can occur as a factor of $S_{k,k-2i}$ since such a square would be a factor of $W_k$ and therefore not new. Neither can a new square containing the letter $k+1$ since this occurs only once in $W_{k+1}$. So any new square must contain the last letter of $S_{k,k-2i}$ for some $i$.  This letter is always $k$ and no square in $S_k$ contains $k$ since $k$ occurs only once in $S_k$. It follows that the squares described above are indeed new.

Now we show there are no other new squares in $W_{k+1}$.  If there were others they would contain at least two copies of $k$.  Such a square can not contain more than two copies since the distances between consecutive occurrences of $k$ are different. Thus a new square must contain the $k$ at the end of $S_{k,k-2i}$ for some $i$, and the $k$ at the start of $S_{k,k-2(i+1)}$.  If the square is not $S_{k,k-2i}^2$ then its first half contains the $k$ at the end of $S_{k,k-2i}$ and the letter following that $k$ which is $2i$.  For it to be a square $2i$ should equal the letter at the end of $S_{k,k-2i+2}$ which is $i+1$ or $k$ and does not equal $2i$.  Therefore we cannot have such a square and the only new squares are those noted above.  This completes the proof of part (a).

The number of distinct squares in $W_{k+1}$ is then $\sum_{i=0}^k \lfloor i/2 \rfloor$ which equals $\lfloor k/2 \rfloor \lceil k/2 \rceil$. This is part (b).
\end{proof}
\end{theorem}
Recall that a \textit{run} is a periodic factor whose length is at least twice its period.  The computation of runs is important algorithmically \cite{KK00,M89}.  The maximum number of runs in any word of length $n$ is  denoted  $\rho(n)$: recently Bannai \emph{et al.}~\cite{BIINTT14} proved the long-standing conjecture that $\rho(n)<n$ for all $n$. If $X[1..2p+l]$ is a run with period $p$ and $l>0$ then $X[1..2p]$ and $X[2..2p+1]$ are both squares.  We saw in the proof of Theorem~\ref{T:distinct-squares} that such pairs of squares do not exist in $W_k$.  We therefore have the following result.
\begin{corollary} Every run in the ZWW word is a square.
\end{corollary}

\section{Total number of squares in $W_k$} \label{S:total-squares}

Let $T(W_i)$ denote the total number of squares in $W_i$ (i.e., counted according to multiplicity).

\begin{theorem} For all $i \ge 1$, we have $T(W_i)=f_{i}-1$.
\begin{proof}  Clearly,
\begin{equation} \label{e2} T(W_{i+2})=T(W_{i+1})+T(2 \oplus W_{i})+S(W_{i+2})
 \end{equation} where $S(W_{i+2})$ is the number of squares which straddle the boundary between $W_{i+1}$ and $2 \oplus W_i$. We evaluate this in the following claim.\\

\noindent \textbf{Claim.} For $i<3$, we have $S(W_i)=0$, and for $i\ge 3$, we have $S(W_i)=1$.

\noindent \textbf{Proof of Claim.}  Recall that $|W_i|=f_{i+2}$. The cases for $i<3$ are easily checked so we assume $i \ge 3$. To simplify notation we temporarily write $w$ for $W_i$ so that a straddling square must begin in $w[1..f_{i+1}]$ and end in $w[f_{i+1}+1..f_{i+2}]$. We first show that, for such $i$, $S(w)\ge 1$. Note that, for $i\ge 2$,
\begin{eqnarray*}
w &=& W_{i-1} \cdot 2 \oplus W_{i-2}\\
&=& W_{i-2}\cdot 2 \oplus W_{i-3} \cdot 2 \oplus (W_{i-3} \cdot 2 \oplus W_{i-4})\\
&=& W_{i-2}\cdot (2 \oplus W_{i-3})^2 \cdot  4 \oplus W_{i-4}.
\end{eqnarray*} We see that  this contains $(2 \oplus W_{i-3})^2$ which is straddling, so $S(w)\ge 1 $.
For example,
$$W_5=01223 \cdot 234 \cdot 234 \cdot 45.$$ Using (\ref{introe0}) we have
\begin{eqnarray*}
&& w[1..f_{i}]=W_{i-2}\\
&& w[f_i+1..f_i+f_{i-1}]=w[f_i+1..f_{i+1}]=2 \oplus W_{i-3}\\
&& w[f_{i+1}+1..f_{i+1}+f_{i-1}]=2 \oplus W_{i-3}\\
&& w[f_{i+1}+f_{i-1}+1..f_{i+1}+f_{i-1}+f_{i-2}]=w[f_{i+1}+f_{i-1}+1..f_{i+2}]=4 \oplus W_{i-4}
\end{eqnarray*}
From \eqref{introe0} and Lemma \ref{l1}, we see that $$w[f_{i+1}]=w[f_{i+1}+f_{i-1}]=i-1$$ and there are no occurrences of $i-1$ in $w$ before position $f_{i+1}$. So each half of the Square $(2 \oplus W_{i-3})^2$ ends in $i-1$. Thus any straddling square must contain both these occurrences of $i-1$. If it contained more than two occurrences of $i-1$ it would have to contain both of these and its first half would contain $w[f_{i+1}..f_{i+1}+f_{i-1}]$ and so its period would be at least $f_{i-1}+1$. But its second half would be contained in $w[f_{i+1}+f_{i-1}+1..f_{i+2}]$ which has length $$f_{i+2}-f_{i+1}-f_{i-1}=f_{i-2}$$ which is less than $f_{i-1}+1$ so no such square can exist.  We conclude that any straddling square in $w$ contains exactly these two instances of $i-1$.  Its period must therefore be  $f_{i-1}$.                                                                                                                        We are therefore asking whether there is a square of length $2f_{i-1}$ in $(2 \oplus W_{i-3})^2 \cdot  4 \oplus W_{i-4}$ which is not a prefix.  If there were such a square there would also be one in $ W_{i-3}^2 \cdot  2 \oplus W_{i-4}$.  This would necessarily contain the 0 at the start of the second $W_{i-3}$ but this is the only 0 in the word other than the initial 0.  We conclude that no such square exists and the only straddling square is $(2 \oplus W_{i-3})^2$.  So for $i\ge 3$, $S(W_i)=1$.

Using this lemma, and noting that $T(2 \oplus W_{i})=T( W_{i})$, equation (\ref{e2}) becomes
\begin{equation} \label{e3} T(W_{i+2})=T(W_{i+1})+T( W_{i})+1
 \end{equation} with $T(W_1)=T(W_2)=0$.  The solution to this is $T(W_i)=f_i-1$.
\end{proof}
\end{theorem}

\section{Lyndon factors} \label{S:Lyndon}

\begin{theorem}
For all $n \geq 0$, $W_n$ is a Lyndon word.
\end{theorem}
\begin{proof}
For each $n \geq 0$, $W_n$ begins with $0$ and contains no other occurrences of $0$ by Lemmas~\ref{lemm-l2} and \ref{l1}; hence $W_n$ is a Lyndon word.
\end{proof}

\begin{theorem} \label{T:Lyndon}  Let $\mathcal{L}_k(W_n)$ denote the number of Lyndon factors beginning with the letter $k$ in $W_n$.
\begin{enumerate}
\item[(i)] For all $n \geq 0$, $\mathcal{L}_0(W_n) = f_{n+2}$.
\item[(ii)] For odd $k \geq 1$ and $n \geq k$, we have $\mathcal{L}_k(W_n) = f_{n+3-k} - 1$.
\item[(iii)] For all $n \geq 2$, $\mathcal{L}_2(W_n) = \sum_{j=0}^{n-2}\left\{\sum_{i=j}^{n-2}f_{i+2} - (n-2-j)f_{j+2}-f_{j+1}\right\}+1$.
\item[(iv)] For even $k \geq 4$ and $n \geq k$, we have $\mathcal{L}_k(W_n) = \mathcal{L}_2(W_{n+2-k})$. \\
\end{enumerate}
\end{theorem}
\begin{proof} ~
\begin{enumerate}
\item[(i)]  For each $n \geq 0$, $W_n$ begins with $0$ and contains no other occurrences of $0$ (by Lemmas~\ref{lemm-l2} and~\ref{l1}), so each of the $f_{n+2} = |W_n|$ prefixes of $W_n$ is a Lyndon word. Hence $\mathcal{L}_0(W_n) = f_{n+2}$.
\item[(ii)] For each $n \geq 1$, $W_n$ begins with $01$ and contains no other occurrences of $0$ and $1$ (by Lemmas~\ref{lemm-l2} and~\ref{l1}), so each of the $f_{n+2}  - 1$ prefixes of $0^{-1}W_n$ is a  Lyndon word. Hence $\mathcal{L}_1(W_n) = f_{n+2} - 1$. For odd $k \ge 3$, it follows from Lemma~\ref{l1} that $\mathcal{L}_{k}(W_n)=0$ for $n < k$, and for $n \geq k$, $\mathcal{L}_k(W_n)$ is equal to the number of Lyndon factors beginning with $1$ in $W_{n-(k-1)}$, i.e., $\mathcal{L}_k(W_n) = f_{n+2-(k-1)} - 1$.
\item[(iii)] This part is trickier than the others and is proved separately in \S\ref{SS:Lyndon2} below.
\item[(iv)] For even $k \geq 4$, $\mathcal{L}_k(W_n) = 0$ for $n < k$ (by Lemma~\ref{l1}), and for $n \ge k$, it follows from Lemma~\ref{lemm-l2} that $\mathcal{L}_k(W_n)$ is equal to the number of Lyndon factors beginning with $2$ in $W_{n-(k-2)}$, i.e., $\mathcal{L}_k(W_n) = \mathcal{L}_2(W_{n-(k-2)})$.
\end{enumerate}
\end{proof}

\subsection{Proof of part (iii) of Theorem~\ref{T:Lyndon}} \label{SS:Lyndon2}

For $i=2,3,4,5,6,7,8$, we observe that $\mathcal{L}_2(W_i)$ is equal to $1, 3, 7, 18, 42, 93, 195$, respectively. These numbers look mysterious, without any obvious connection to the Fibonacci numbers, as one might expect. But, indeed, the Fibonacci numbers are involved in the formula for $\mathcal{L}_2(W_i)$ that we determine below.

In what follows, we set  \begin{equation} \label{eq1}
X_i=2 \oplus W_i \quad \mbox{for \quad $i \ge 0$}
\end{equation}so that
\begin{eqnarray*}
&X_0&=2\\
&X_1&=23\\
&X_2&=234\\
&X_3&=23445
\end{eqnarray*}

\begin{lemma} \label{lyndon-l2} For $k\ge 2$, we have
$$W_k=01X_0X_1\cdots X_{k-2}.$$
\end{lemma}
\begin{proof}
This is immediate from Lemma \ref{intro l2}.
\end{proof}
\noindent \textbf{Example:} $W_5=01\cdot 2\cdot 23\cdot 234\cdot 23445=01\cdot X_0\cdot X_1\cdot X_2\cdot X_3.$

\begin{lemma} \label{l3} A factor of the ZWW word which begins with 2 is Lyndon unless it is bordered.
\begin{proof}
For the sake of contradiction suppose that $u$ is an unbordered factor of the $ZWW$ word with $u[1]=2$ which is not Lyndon. Since it is not Lyndon, it has a suffix $s$ which is lexicographically  less than $u$.  Let $p$ be the prefix of $u$ of length $|s|$.  Since $u$ is unbordered  $p \not = s$ and we must have $s$ less than $p$.  Since $p$ begins with 2 so must $s$ and each therefore begins at the start of $X_i$ for some $i$.  Say that $p$ is a prefix of $X_iX_{i+1}\cdots$ and $s$ is a prefix of $X_{i+j}X_{i+j+1}\cdots$. By Lemma \ref{lyndon-l2}, $X_i$ is a prefix of $X_{i+j}$.  Let $X'$ be the length $|X_{i+j}|-|X_i|$ suffix of $X_{i+j}$.  Then we have a prefix of $X'X_{i+j+1}\cdots$ which is lexicographically less than a prefix of $X_{i+1}\cdots$.  But that is impossible since $X_{i+1}\cdots$ begins with 2 and $X'X_{i+j+1}\cdots$ begins with something larger than 2. This contradiction completes the proof.
\end{proof}
\end{lemma}
\begin{theorem} The number of Lyndon words beginning with 2 in $W_i$ is $$\sum_{j=0}^{i-2}\left\{\sum_{k=j}^{i-2}f_{k+2} - (i-2-j)f_{j+2}-f_{j+1}\right\}+1.$$
\end{theorem}
\begin{proof}  Each Lyndon word beginning with 2 in $W_i$ must begin at the the beginning of $X_j$ for some $j$ with $0\le j \le i-2.$  We first calculate the number of such words for a particular value of $j$ and then sum over~$j$.

Let $u$ be a word beginning at the start of $X_j$ and ending in $W_i$. By (\ref{eq1}) and Lemma \ref{lyndon-l2}, we see that each pair of consecutive 2s in $W_i$ is separated by a different distance.  It follows that a border of $u$ cannot contain more than one 2.  It therefore contains exactly one 2 and has length at most $|X_j|$.  By Lemma \ref{lemm-l2}, $X_j$ is a prefix of $X_k$ for all $k$ greater than $j$ so $u$ has a border if and only if $u$ ends inside a length $|X_j|$ prefix of some $X_k$ for $j+1 \le k \le i-2$.  There are thus $(i-2-j)|X_j|$ words $u$ which are not Lyndon.  For $j>0$ the total number of words in $W_i$ that begin at the start of $X_j$ and finish after the length $|X_{j-1}|$ prefix of $X_j$ is
\begin{eqnarray*}
\sum_{k=j}^{i-2}f_{k+2} -|X_{j-1}|. 
\end{eqnarray*} The length $|X_{j-1}|$ prefix is excluded because any Lyndon word in it will duplicate one in $X_{j-1}$. Subtracting the number of non-Lyndon words from this gives the total number of Lyndon words in $W_i$ that begin at the start of $X_j$:
\begin{eqnarray*}
\sum_{k=j}^{i-2}f_{k+2} - (i-2-j)|X_j|-|X_{j-1}|. 
\end{eqnarray*}

To get the total number of Lyndon words in $W_i$ we sum this over $j$.  In doing so we replace $|X_{j-1}|$ with $f_{j+1}$ when $j>0$.  When $j=0$ there is nothing to subtract, however to make the formula nicer we subtract $f_{j+1}$ in this case as well and compensate by adding 1 to the final formula. The total number of Lyndon words beginning with 2 in $W_i$ is therefore

\begin{eqnarray*}
\sum_{j=0}^{i-2}\left\{\sum_{k=j}^{i-2}f_{k+2} - (i-2-j)f_{j+2}-f_{j+1}\right\}+1.
\end{eqnarray*}

\end{proof}

\section{Lyndon array of $W_k$}
\label{sect-lyndon-array}

The {\it Lyndon array} $\lambda = \lambda_X[1..n]$
(equivalently, $\L = \L_X[1..n]$)
of a given non-empty word $X = X[1..n]$
gives at each position $i$ the length
(equivalently, the end position)
of the longest Lyndon word starting at $X[i]$. For example:

\begin{equation}
\label{ex1}
\begin{array}{rccccccccc}
\scriptstyle 1 & \scriptstyle 2 & \scriptstyle 3 & \scriptstyle 4 & \scriptstyle 5 & \scriptstyle 6 & \scriptstyle 7 & \scriptstyle 8 & \scriptstyle 9 & \scriptstyle 10 \\
X = a & b & a & a & b & a & b & a & a & b \\
\lambda  = 2 & 1 & 5 & 2 & 1 & 2 & 1 & 3 & 2 & 1 \\
\L = 2 & 2 & 7 & 5 & 5 & 7 & 7 & 10 & 10 & 10
\end{array}
\end{equation}
Clearly $\lambda[i] = \L[i] - i + 1$.

The Lyndon array has only recently been clearly defined \cite{FIRS16},
but turns out to have an intimate connection with the suffix array \cite{PST07}
that so far is not well understood:
the suffix array can be computed from $\lambda$ in linear time
using the Next Smaller Value algorithm \cite{HR03,FIRS16};
on the other hand, a sorted version of $\lambda$ is computed
by Phase I of the first non-recursive linear-time
suffix array computation algorithm due to Baier \cite{B16}.
To date nine algorithms, including Baier's,
have been discovered to compute the Lyndon array of $X$,
of which the fastest in practice is apparently a brute force approach
that requires $O(n^2)$ time in the worst case~\cite{FIRS16}.

It is thus perhaps of interest to investigate the Lyndon arrays $\lambda_k/\L_k$
of $W_k$ for given finite $k$.
Trivially, since for all positive $k$,
$W_k[1..2] = 01$, we have $\L_k[1] = \L_k[2] = w_k$ where $w_k = |W_k| = f_{k+2}$.
If for $k \ge 2$, we define $W_{k-1}^{k-2,2} = 2 \oplus W_{k-2}$
to express the suffix of $W_k$ in terms of a prefix of length $w_{k-2}$
of $W_{k-1}$, Lemma~\ref{e1} yields the following.

\begin{remark}
\label{rem-1}
For $k \ge 2$, $W_k = W_{k-1} \cdot W_{k-1}^{k-2,2}$.
\end{remark}
Thus, using $w_{k-2}$, each $W_k$ can be computed
by a direct calculation from $W_{k-1}$.
The first few values, for $k = 2, 3, 4$, are
\begin{eqnarray*}
W_2 &=& W_1\cdot W_1^{0,2} = 01\cdot 2 \\
W_3 &=& W_2\cdot W_2^{1,2} = 012\cdot 23 \\
W_4 &=& W_3\cdot W_3^{2,2} = 01223\cdot 234
\end{eqnarray*}
The following results are then immediate.
\begin{remark}
\label{rem-2}
{\ }
\begin{itemize}
\item[(a)]
For $k \ge 2$,
the digit 2 occurs \emph{only} as the first digit of $W_{k-1}^{k-2,2}$.
\item[(b)]
Since $W_{k-1}^{k-2,2}$ is a prefix of $W_k^{k-1,2}$, we have $W_{k-1}^{k-2,2} < W_k^{k-1,2}$.
\end{itemize}
\end{remark}
Thus for $i \ge 2$, $W_k[i] = 2 \Rightarrow \L[i] = w_k$.
Otherwise, $W_k[i] > 2$, and so for all such $i$,
$\lambda_k[i]$ must take the same values assumed by the corresponding
positions in $W_{k-1}$.
More precisely:
\begin{lemma}
\label{lemm-lyndon}
For $k \ge 1$,
\begin{itemize}
\item[(a)]
$1 \le i \le w_{k-1} \Rightarrow \lambda_k[i] = \lambda_{k-1}[i]$;
\item[(b)]
$w_{k-1} + 1 \le i \le w_k \Rightarrow \lambda_k[i] = \lambda_{k-1}[i - w_{k-1}]$.
\end{itemize}
\end{lemma}

Based on this result,
the algorithm shown in Figure~\ref{fig-la} computes $\L_{W_k}$
in time $O(w_k)$; implementing the statements in square brackets
also yields the computation of $W_k$. Apart from storage for $W_k$ and $\L$, the algorithm requires only constant space.

\begin{figure}[htb!]
{\leftskip=1.0in\obeylines\sfcode`;=3000
\bproc LA $(k)$
\com{Suppose that $w_k = f_{k+2}$ has been precomputed.}
$\big[W[1..2] \la 01\big];\ \L[1..2] \la w_kw_k;\ w_{-1} \la 2;\ w_{-2} \la 1$
\bfor $j \la 2$ \bto $k$ \bdo
\qq $w \la w_{-1}\+ w_{-2}$
\qq \bfor $i \la w_{-1}\+ 1$ \bto $w$ \bdo
\qq\qq $\big[W[i] \la W[i\m w_{-1}]\+ 2\big]$
\qq\qq \bif $i = w_{-1}\+ 1$ \bthen
\qq\qq\qq $\L[i] \la w_k$
\qq\qq \belsif $\L[i\m w_{-1}] = w_k$ \bthen
\qq\qq\qq $\L[i] \la w$
\qq\qq \belse
\qq\qq\qq $\L[i] \la \L[i\m w_{-1}]\+ w_{-1}$
\qq $w_{-2} \la w_{-1};\ w_{-1} \la w$
}
\caption{Compute $W_k$ and its Lyndon array $\L_{W_k}$, $k \ge 0$.}
\label{fig-la}
\end{figure}

We note that Algorithm LA \textit{uses only elementary methods} to compute $W_k$ and its Lyndon array in linear time using only constant additional space; all general-purpose linear-time Lyndon array algorithms (e.g., \cite{B16, HR03}) require advanced data structures (suffix array, etc.) and $O(w_k)$ additional space.


\section{Concluding Remarks}

In this paper, we have counted the number of squares, palindromes, and Lyndon factors occurring in the finite building blocks, $W_n$, of the ZWW word. The following table shows the number of squares (distinct and total) in the words $W_n$ and the original finite Fibonacci words $F_n$.

\medskip
\begin{table}[htb!]
\begin{tabular}{lcc}
\hline
~ &Number of Distinct Squares &\qquad \qquad Total Number of Squares \\
\hline
$W_n$ & $\lfloor (n-1)/2 \rfloor\lceil(n-1)/2 \rceil$  &\qquad \qquad $f_n-1$ \\
$F_n$ &$2(f_n-1)$ &\qquad \qquad $\frac{4}{5}(n+1)f_{n+2} - \frac{2}{5}(n+7)f_{n+1} - 4f_n + n + 2$ \\
\hline 
\end{tabular} \medskip
\caption{Distinct and total numbers of squares in the words $W_n$ and $F_n$.}
\end{table}

The formulas for $F_n$ were determined by Fraenkel and Simpson~\cite{FS99}. We see that the numbers of squares (distinct and total) grow slower for the $W_n$ words compared to the Fibonacci words. Moreover, the ZWW word does not contain any higher powers whereas the Fibonacci word contains cubes (e.g., $(010)^3$ is a factor of $F_6$). The $W_n$ words also contain much fewer palindromic factors --- we showed that $W_n$ contains $\lfloor 5n/2\rfloor - 2$ distinct palindromes, whereas the finite Fibonacci words are well known to be `rich' in palindromes in the sense that a new palindrome is introduced at each position, i.e., the number of distinct non-empty palindromes in $F_n$ is $|F_n| = f_{n+2}$ for all $n \geq 0$ (a word of length $n$ contains at most $n$ distinct non-empty palindromes~\cite{DJP}). 

Intuitively, one would expect lower numbers of patterns such as squares and palindromes to occur in the ZWW word because of the introduction of a new letter $n$ at the end of each $W_n$. The well-studied Fibonacci word $F$, on the other hand, has a rich structure, being a special example of a \textit{Sturmian word}; such words are 'almost periodic' in the sense that they are aperiodic words of minimal complexity (for more details, see \cite[Ch.\ 2]{mL02alge}). 

Lastly, we note that, unlike what we have observed for squares and palindromes, Lyndon factors are more plentiful in $W_n$ compared to $F_n$. This is because each $W_n$ is itself a Lyndon word, which has the Lyndon words $W_i$, $0 \leq i \leq n$, as factors. By contrast, none of the $F_n$ are Lyndon words, but each such word has a unique circular shift (conjugate) which is Lyndon --- called the \textit{Fibonacci Lyndon word} of length $|F_n| = f_{n+2}$ --- and contains the \textit{minimum} number of Lyndon factors (specifically, $n+2$ of them) over all Lyndon words of length at least $f_{n+2}$ (see \cite{kS14lynd}).

 \end{document}